\documentclass[11pt, final]{article}
\usepackage{a4}

\usepackage{amsmath}
\usepackage{amssymb}
\usepackage{amsxtra}
\usepackage{amsthm}
\usepackage{bbm}			
\usepackage{cite}

\usepackage[left=3.5cm,right=3.5cm,top=3cm,bottom=3cm]{geometry}

\usepackage[plainpages=false,pdfpagelabels]{hyperref} 

\usepackage{showkeys}
\usepackage[latin1]{inputenc}       
\usepackage[T1]{fontenc}        
\usepackage{lmodern}            
\usepackage{graphicx}
\usepackage{subfigure}		
\usepackage{booktabs}		
\usepackage{multirow}
\usepackage{enumitem}		
\usepackage{listings} 	

\usepackage{colortbl}
\definecolor{dunkelgrau}{rgb}{0.8,0.8,0.8}
\definecolor{hellgrau}{rgb}{0.9,0.9,0.9}

\newcommand{\R}{\ensuremath{\mathbb{R}}}

\theoremstyle{plain}
\newtheorem{theorem}{Theorem}

\newtheorem{proposition}[theorem]{Proposition}

\theoremstyle{definition}
\newtheorem{remark}[theorem]{Remark}

\DeclareMathOperator*\dom{dom}%
\DeclareMathOperator*\ri{ri}%
\DeclareMathOperator*\argmin{arg\,min}%
\DeclareMathOperator*\cl{cl}%

\title{A double smoothing technique for solving unconstrained nondifferentiable convex optimization problems}

\author{Radu Ioan Bo\c t
\thanks {Faculty of Mathematics, Chemnitz University of Technology, D-09107 Chemnitz, Germany, e-mail: radu.bot@mathematik.tu-chemnitz.de. Research partially supported by DFG (German Research Foundation), project BO 2516/4-1.}
\and Christopher Hendrich
\thanks{Faculty of Mathematics, Chemnitz University of Technology, D-09107 Chemnitz, Germany, e-mail: christopher.hendrich@mathematik.tu-chemnitz.de.}
}
\date{\today}

\begin{document}
\maketitle

{\bf Abstract.} The aim of this paper is to develop an efficient algorithm for solving a class of unconstrained nondifferentiable convex optimization problems in finite dimensional spaces. To this end we formulate first its Fenchel dual problem and regularize it in two steps into a differentiable strongly convex one with Lipschitz continuous gradient. The doubly regularized dual problem is then solved via a fast gradient method with the aim of accelerating the resulting convergence scheme. The theoretical results are finally applied to an $l_1$ regularization problem arising in image processing.



\section{Introduction}\label{sectionIntro}

In this paper we are interested in solving a specific class of unconstrained convex optimization problems in finite dimensional spaces. Generally, when characterizing optimality, the convexity allows to make use of powerful results in convex analysis, separation theorems and the (Fenchel) conjugate theory here included (see \cite{Zalinescu02,Rockafellar70,BauschkeCombettes11}). In convex optimization these are the ingredients for assigning a dual optimization problem via the perturbation approach to a primal one. When strong duality holds, solving the dual problem instead is a natural way to obtain an optimal solution to the primal problem, too. As weak duality is always fulfilled, for guaranteeing strong duality, so-called regularity conditions are needed (see, for example, \cite{Bot10,BotGradWanka09,Zalinescu02}).

When considering an unconstrained convex and differentiable minimization problem, there are already plenty of promising methods available (such as the steepest descent method, Newton's method or, in an appropriate setting, fast gradient methods, see \cite{Nesterov04}) for solving it. However, a lot of situations occur when the objective function of the optimization problem to be solved is nondifferentiable. Therefore, the convex subdifferential is used instead, not only as a tool for theoretically characterizing optimality, but also as the counterpart of the gradient in different numerical methods. However, the classical methods which solve unconstrained convex and nondifferentiable minimization problems have a rather slow convergence.

The aim of this paper is to develop in finite dimensional spaces an efficient algorithm for solving an unconstrained optimization problem having as objective the sum of a convex function with the composition of another convex function with a linear operator. To this end we are not relying on subgradient schemes, since their complexity can not be better than $O\left(\frac{1}{\epsilon^2}\right)$ iterations, where $\epsilon > 0$ is the desired accuracy for the objective value (see \cite{Nesterov04}). Instead, we show that it is possible to solve the corresponding Fenchel dual problem efficiently and to reconstruct in this way an approximately optimal solution to the primal one. To this end we make use of a double smoothing technique, in fact a generalization of the double smoothing approach employed by Devolder, Glineur and Nesterov in \cite{NesterovDoubleSmooth} and \cite{NesterovDoubleSmoothInfinite} for a special class of convex constrained optimization problems. This technique makes use of the structure of the dual problem and assumes the regularization of its objective function into a differentiable strongly convex one with Lipschitz continuous gradient. The regularized dual is then solved by a fast gradient method and this gives rise to a sequence of dual variables which solve the non-regularized dual objective in $O\left(\frac{1}{\epsilon} \ln\left( \frac{1}{\epsilon}\right)\right)$ iterations. In addition, the norm of the gradient of the objective of the regularized dual decreases by the same rate of convergence, a fact which is crucial in view of reconstructing an approximately optimal solution to the primal optimization problem.

The structure of the paper is the following. In the forthcoming section we introduce the class of convex optimization problems which we deal with throughout this paper, provide its Fenchel dual optimization problem and discuss some duality issues. In Section \ref{sectionDS} we apply the smoothing technique introduced in \cite{NesterovExcessiveGap05, NesterovSmoothMin05,NesterovSmoothing05} to the dual objective function in order to make it strongly convex and differentiable with Lipschitz continuous gradient. In Section \ref{sectionConvergence} the regularized dual problem is solved via an efficient fast gradient method. Additionally, we investigate the convergence of the dual iterates to an optimal dual solution with  a given accuracy and show how to reconstruct from it an approximately optimal primal solution. Finally, in Section \ref{sectionExample}, an $l_1$ regularized linear inverse problem is solved via the presented approach and an application in image processing is discussed.

\section{Preliminaries and problem formulation}\label{sectionPrelimFormulation}

In the following we are considering the space $\mathbb{R}^n$ endowed with the the Euclidean topology, i.\,e. $\left\| x \right\| = \sqrt{\left\langle x, x \right\rangle} = \sqrt{ x^T x }$ for all $x \in \mathbb{R}^n$. By $\mathbbm{1}^n$ we denote the vector in $\R^n$ with all entries equal to $1$. For a subset $C$ of $\mathbb{R}^n$ we denote by $\cl C$ and $\ri C$ its \textit{closure} and \textit{relative interior}, respectively. The indicator function of the set $C$ is the function $\delta_C : \R^n \rightarrow  \overline{\mathbb{R}} := \mathbb{R} \cup \left\{ \pm \infty \right\}$ defined by $\delta_C(x) = 0$ for $x \in C$ and $\delta_C(x) = +\infty$, otherwise. For a function $f: \mathbb{R}^n \rightarrow \overline{\mathbb{R}}$ we denote by $\dom f := \left\{ x \in \mathbb{R}^n : f(x) < +\infty \right\}$ its \textit{effective domain}. We call $f$ \textit{proper} if $\dom f \neq \emptyset$ and $f(x)>-\infty$ for all $x \in \mathbb{R}^n$. The \textit{conjugate function} of $f$ is $f^*:\mathbb{R}^n \rightarrow \overline{\mathbb{R}}$, $f^*(p)=\sup{\left\{ \left\langle p,x \right\rangle -f(x) : x\in\mathbb{R}^n \right\}}$ for all $p \in \mathbb{R}^n$. The \textit{biconjugate function} of $f$ is $f^{**} : \R^n \rightarrow \overline \R$, $f^{**}(x) = \sup{\left\{ \left\langle x,p \right\rangle -f^*(p) : p\in\mathbb{R}^n \right\}}$ and, when $f$ is proper, convex and lower semicontinuous, according to the Fenchel-Moreau Theorem, one has $f=f^{**}$. The \textit{(convex) subdifferential} of the function $f$ at $x \in \R^n$ is the set $\partial f(x) = \{p \in \R^n : f(y) - f(x) \geq p^T(y-x) \ \forall y \in \R^n\}$, if $f(x) \in \R$, and is taken to be the empty set, otherwise. For a linear operator $A: \mathbb{R}^n \rightarrow \mathbb{R}^m$, the operator $A^*: \mathbb{R}^m \rightarrow \mathbb{R}^n$ is the \textit{adjoint operator} of $A$ and is defined by $\left\langle A^*y,x  \right\rangle = \left\langle y,Ax  \right\rangle$ for all $x \in \mathbb{R}^n$ and all $y \in \mathbb{R}^m$.

For a nonempty, convex and closed set $C \subseteq \R^n$ we consider the \textit{projection operator} $\mathcal{P}_C : \mathbb{R}^n \rightarrow C$ defined as $x \mapsto \argmin_{z\in C}\left\| x-z \right\|$. Having two proper functions $f,\,g : \mathbb{R}^n \rightarrow \overline{\mathbb{R}}$, their \textit{infimal convolution} is defined by $f \Box g : \mathbb{R}^n \rightarrow \overline{\mathbb{R}}$, $(f \Box g) (x) = \inf_{y \in \mathbb{R}^n}\left\{ f(y) + g(x-y) \right\} $ for all $x \in \mathbb{R}^n$. The \textit{Moreau envelope} of the function $f : \R^n \rightarrow \overline \R$ of parameter $\gamma > 0$ is defined as the infimal convolution
$$^{\gamma}f(x) := f \Box \left(\frac{1}{2\gamma}\left\| \cdot \right\|^2\right)(x) = \inf_{y \in \R^n} \left \{f(y) + \frac{1}{2\gamma}\|x-y\|^2 \right\} \ \forall x \in \R^n.$$
We say that the function $f : \R^n \rightarrow \overline \R$ is \textit{strongly convex} with parameter $\rho > 0$ if for all $x,y \in \R^n$ and all $\lambda \in \left(0,1\right)$ it holds
$$f(\lambda x + (1-\lambda)y) \leq \lambda f(x) + (1-\lambda)f(y) - \frac{\rho}{2}\lambda(1-\lambda)\|x-y\|^2.$$

In this work we are dealing with optimization problems of the type
\begin{equation}\label{opt-problem:primal}
\hspace{-1.8cm}(P) \quad \quad \inf_{x \in \mathbb{R}^n}{\left\{f(x)+g(Ax)\right\}},
\end{equation}
where $f:\mathbb{R}^n \rightarrow \overline{\mathbb{R}}$ and $g:\mathbb{R}^m \rightarrow \overline{\mathbb{R}}$ are proper, convex and lower semicontinuous functions and $A:\mathbb{R}^n \rightarrow \mathbb{R}^m$ is a linear operator fulfilling $A(\dom f) \cap \dom g \neq \emptyset$. Furthermore, we assume that $\dom f$ and $\dom g$ are \textit{bounded}.

\begin{remark}
\label{double-smooth-remark-domf-or-domg-unbounded}
The assumption that $\dom f$ and $\dom g$ are bounded can be weakened in the sense that it is sufficient to assume that $\dom f$ is bounded. In this situation,
in the formulation of $(P)$ the function $g$ can be replaced by $g + \delta_{\cl(A(\dom f))}$, which is a proper, convex and lower semicontinuous function with bounded effective domain.

On the other hand, one should also notice that the counterparts of the assumptions considered in \cite{NesterovDoubleSmooth, NesterovDoubleSmoothInfinite} in our setting would ask for closedness for the effective domains of the functions $f$ and $g$, too. However, we will be able to employ the double smoothing technique for $(P)$ without being obliged to impose this assumption.
\end{remark}

According to \cite{Bot10,BotGradWanka09}, the Fenchel dual problem to $(P)$ is nothing else than
\begin{equation}\label{opt-problem:dual}
\hspace{-1.8cm}(D) \quad \quad \sup_{p \in \mathbb{R}^m}{ \left\{ -f^*(A^*p)-g^*(-p) \right\} },
\end{equation}
where $f^*:\mathbb{R}^n \rightarrow \overline{\mathbb{R}}$ and $g^*:\mathbb{R}^m \rightarrow \overline{\mathbb{R}}$ denote the conjugate functions of $f$ and $g$, respectively. We denote the optimal objective values of the optimization problems $(P)$ and $(D)$ by $v(P)$ and $v(D)$, respectively.

The conjugate functions of $f$ and $g$ can be written as
\begin{align*}
	f^*(q)
	= \sup_{x \in \dom f}{ \left\{ \left\langle  q,x \right\rangle  -f(x) \right\} }
	= -\inf_{x \in \dom f}{ \left\{ \left\langle -q,x \right\rangle  +f(x) \right\} } \ \forall q \in \R^n
\end{align*}
and
\begin{align*}
	g^*(p)
	= \sup_{x \in \dom g}{ \left\{ \left\langle  p,x \right\rangle  -g(x) \right\} }
	= -\inf_{x \in \dom g}{ \left\{ \left\langle  -p,x \right\rangle  +g(x) \right\} } \ \forall p \in \R^m,
\end{align*}
respectively. In the framework considered above, according to \cite[Proposition A.8]{Bertsekas99}, the optimization problems arising in the formulation of $f^*(q)$ for all $q \in \R^n$ and $g^*(p)$ for all $p \in \R^m$ are solvable, fact which implies that $\dom f^* = \mathbb{R}^n$ and $\dom g^* = \mathbb{R}^m$, respectively.

By writing the dual problem $(D)$ equivalently as the infimum optimization problem
$$\inf_{p \in \R^m} \{f^*(A^*p) + g^*(-p)\},$$
one can easily see that the Fenchel dual problem of the latter is
$$\sup_{x \in \R^n} \{-f^{**}(x) - g^{**}(Ax)\}, $$
which, by the Fenchel-Moreau Theorem, is nothing else than
$$\sup_{x \in \R^n} \{-f(x) - g(Ax)\}.$$
In order to guarantee strong duality for this primal-dual pair it is sufficient to ensure that (see, for instance, \cite{Bot10}) $0 \in \ri(A^*(\dom g^*) + \dom f^*)$. As $f^*$ has full domain, this regularity condition is automatically fulfilled, which means that $v(D) = v(P)$ and the primal optimization problem $(P)$ has an optimal solution. Due to the fact that $f$ and $g$ are proper and $A(\dom f) \cap \dom g \neq \emptyset$, this further implies $v(D)=v(P) \in \mathbb{R}$. Later we will assume that the dual problem $(D)$ has an optimal solution, too, and that an upper bound of its norm is known.

Denote by $\theta:\mathbb{R}^m \rightarrow \mathbb{R}$, $\theta(p)=f^*(A^*p)+g^*(-p)$, the objective function of $(D)$. Hence, the latter can be equivalently written as
\begin{equation}
\label{opt-problem:dual-in-minimization-form}
\hspace{-1.8cm}(D) \quad \quad -\inf_{p \in \mathbb{R}^m}{ \theta(p) }.
\end{equation}
Since in general we can neither guarantee the smoothness of $p \mapsto f^*(A^*p)$ nor of $p\mapsto g^*(-p)$, the dual problem $(D)$ is a nondifferentiable convex optimization problem. Our goal is to solve this problem efficiently and to obtain from here an optimal solution to $(P)$. To this end, we are not relying on subgradient-type schemes, due to their slow rates of convergence equal to $O\left(\frac{1}{\epsilon^2}\right)$, but we are applying instead some smoothing techniques introduced in \cite{NesterovExcessiveGap05,NesterovSmoothMin05,NesterovSmoothing05}. More precisely, we regularize first the functions $p \mapsto f^*(A^*p)$ and $p\mapsto g^*(-p)$, by taking into account the definitions of the two conjugates, in order to obtain a smooth approximation of the objective of \eqref{opt-problem:dual-in-minimization-form} with a Lipschitz continuous gradient. Then we solve the regularized dual problem by making use of a fast gradient method (see \cite{NesterovSmoothMin05}) and generate in this way a sequence of dual variables which approximately solves the problem $(D)$ with a rate of convergence of $O\left(\frac{1}{\epsilon} \right)$. Since similar properties cannot be ensured for the primal optimization problem $(P)$, the solving of this problem being actually our goal, we apply a second regularization  to the objective function of \eqref{opt-problem:dual-in-minimization-form}. This will allow us to make use of a fast gradient method for smooth and strongly convex functions given in \cite{Nesterov04} for solving the regularized dual, which implicitly will solve both the dual problem $(D)$ and the primal problem $(P)$ approximately in $O\left(\frac{1}{\epsilon}\ln\left(\frac{1}{\epsilon}\right)\right)$ iterations.

\section{The double smoothing approach}\label{sectionDS}

\subsection{First smoothing}\label{subsectionFirstSmooth}
For a positive real number $\rho > 0$ the function $p \mapsto f^*(A^*p) = \sup_{x \in \mathbb{R}^n}{ \left\{ \left\langle  A^*p,x \right\rangle  -f(x) \right\}}$ can be approximated by
\begin{equation}
	\label{opt-problem:f*rho}
	f_{\rho}^*(A^*p) = \sup_{x \in \mathbb{R}^n}{ \left\{ \left\langle  A^*p,x \right\rangle  -f(x) - \frac{\rho}{2} \left\| x \right\|^2 \right\} },
\end{equation}
while, given $\mu > 0$, the function $p \mapsto g^*(-p) = \sup_{x \in \R^n}{ \left\{ \left\langle  -p,x \right\rangle  -g(x) \right\} }$ can be approximated by
\begin{equation}
	\label{opt-problem:g*mu}
	g_{\mu}^*(-p)	= \sup_{x \in \mathbb{R}^m}{ \left\{ \left\langle  -p,x \right\rangle  -g(x) - \frac{\mu}{2} \left\| x \right\|^2 \right\} }.
\end{equation}
For each $p \in \R^m$ the maximization problems which occur in the formulations of $f_{\rho}^*(A^*p)$ and $g_{\mu}^*(-p)$ have unique solution (see, for instance, \cite[Proposition A.8 and Proposition B.10]{Bertsekas99}), since their objectives are proper, strongly concave (see \cite[Proposition B.1.1.2]{HirLem01}) and upper semicontinuous functions.

In order to determine  the gradient of the functions $p\mapsto f^*(A^*p)$ and $p \mapsto g^*(-p)$, we are going to make use of the Moreau envelope of the functions $f$ and $g$, respectively. Indeed, for all $p \in \mathbb{R}^m$ we have
\begin{align*}
	-f_{\rho}^*(A^*p) &= -\sup_{x \in \mathbb{R}^n}{ \left\{ \left\langle  A^*p,x \right\rangle  -f(x) -\frac{\rho}{2} \left\| x \right\|^2 \right\}} \\
	&= \inf_{x \in \mathbb{R}^n}{ \left\{ -\left\langle  A^*p,x \right\rangle  +f(x) +\frac{\rho}{2} \left\| x \right\|^2 \right\}} \\
	&= \inf_{x \in \mathbb{R}^n}{ \left\{ f(x)  +\frac{\rho}{2} \left\| \frac{A^*p}{\rho} -x \right\|^2 \right\}} -\frac{\left\| A^*p\right\|^2}{2\rho}
	= {}^{\frac{1}{\rho}}f\left(\frac{A^*p}{\rho}\right) -\frac{\left\| A^*p\right\|^2}{2\rho}.
\end{align*}
As the Moreau envelope is continuously differentiable (see \cite[Proposition 12.29]{BauschkeCombettes11}), $p \mapsto -f_{\rho}^*(A^*p)$ is continuously differentiable, as well, and it holds for all $p \in \R^m$
\begin{align*}
	- \nabla (f_{\rho}^*\circ A^*)(p) = \frac{A}{\rho}\, \nabla\, {}^{\frac{1}{\rho}}f\left(\frac{A^*p}{\rho}\right) -\frac{AA^*p}{\rho}
	= \frac{A}{\rho}\left( \rho \left( \frac{A^*p}{\rho} - x_{\rho,p} \right) \right)  -\frac{AA^*p}{\rho}
	= -Ax_{\rho,p},
\end{align*}
which means that
\begin{align*}
	\nabla (f_{\rho}^*\circ A^*)(p) = Ax_{\rho,p},
\end{align*}
where $x_{\rho,p} \in \R^n$ is the \textit{proximal point} of parameter $\tfrac{1}{\rho}$ of $f$ at $\frac{A^*p}{\rho}$, namely the unique element in $\R^n$ fulfilling
$${}^{\frac{1}{\rho}}f\left(\frac{A^*p}{\rho}\right) =  f(x_{\rho,p})  +\frac{\rho}{2} \left\| \frac{A^*p}{\rho} -x_{\rho,p} \right\|^2.$$
By taking into account the nonexpansiveness of the proximal point mapping (see \cite[Proposition 12.27]{BauschkeCombettes11}), for $p,q \in \mathbb{R}^m$ it holds
\begin{align*}
	\left\| \nabla (f_{\rho}^*\circ A^*)(p) - \nabla (f_{\rho}^*\circ A^*)(q) \right\| & = \left\| Ax_{\rho,p} - Ax_{\rho,q} \right\| \leq \left\| A \right\| \left\| x_{\rho,p} - x_{\rho,q} \right\| \\
	& \leq \left\| A \right\| \left\| \frac{A^*p}{\rho} - \frac{A^*q}{\rho} \right\| \leq \frac{\left\| A \right\|^2}{\rho} \left\| p - q \right\|,
\end{align*}
thus $\frac{\left\| A \right\|^2}{\rho}$ is the Lipschitz constant of $p \mapsto \nabla (f_{\rho}^*\circ A^*)(p)$.

For the function $p \mapsto g^*(-p)$ one can proceed analogously. For all $p \in \R^m$ one has
\begin{align*}
	-g_{\mu}^*(-p) &= \inf_{x \in \mathbb{R}^m}{ \left\{ g(x)  +\frac{\mu}{2} \left\| -\frac{p}{\mu} -x \right\|^2 \right\}} -\frac{\left\| p\right\|^2}{2\mu}
	= {}^{\frac{1}{\mu}}g\left(-\frac{p}{\mu}\right) -\frac{\left\| p\right\|^2}{2\mu},
\end{align*}
which is a continuously differentiable function such that
\begin{align*}
	- \nabla g_{\mu}^*(-\cdot)(p) = -\frac{1}{\mu}\, \nabla\, {}^{\frac{1}{\mu}}g\left(-\frac{p}{\mu}\right) -\frac{p}{\mu}
	= -\frac{1}{\mu}\left( \mu \left( -\frac{p}{\mu} - x_{\mu,p} \right) \right)  -\frac{p}{\mu}
	= x_{\mu,p},
\end{align*}
thus,
\begin{align*}
	\nabla g_{\mu}^*(-\cdot)(p) = -x_{\mu,p},
\end{align*}
where $x_{\mu,p} \in \R^m$ is the \textit{proximal point} of parameter $\tfrac{1}{\mu}$ of $g$ at $-\frac{p}{\mu}$, namely the unique element in $\R^m$ fulfilling
$${}^{\frac{1}{\mu}}g\left(-\frac{p}{\mu}\right) =  g(x_{\mu,p})  +\frac{\mu}{2} \left\| -\frac{p}{\mu} -x_{\mu,p} \right\|^2.$$
For $p,q \in \mathbb{R}^m$ it holds
\begin{align*}
	\left\| \nabla g_{\mu}^*(-\cdot)(p) - \nabla g_{\mu}^*(-\cdot)(q) \right\|
	&= \left\| -x_{\mu,p} + x_{\mu,q} \right\| \leq \left\| -\frac{p}{\mu} + \frac{q}{\mu} \right\| \leq \frac{1}{\mu} \left\| -p + q \right\|,
\end{align*}
so that $\frac{1}{\mu}$ is the Lipschitz constant of $p \mapsto \nabla g_{\mu}^*(-\cdot)(p)$.

\begin{remark}
\label{remarkRegularizationUnnecessary}
If $f$ is strongly convex with parameter $\rho > 0$, there is no need to apply the first regularization for $p \mapsto f^*(A^*p)$, as this function is already differentiable with a Lipschitz continuous gradient having a Lipschitz constant given by $\frac{\left\| A \right\|^2}{\rho}$. The same applies for $p \mapsto g^*(-p)$, if $g$ is strongly convex with parameter $\mu >0$, in this case the Lipschitz constant of its gradient being given by $\frac{1}{\mu}$.
\end{remark}

The constants $D_f:=\sup\left\{ \frac{\left\| x \right\|^2}{2} : x \in \dom f  \right\} $ and $D_g:=\sup \left\{ \frac{\left\| x \right\|^2}{2} : x \in \dom g  \right\}$ will play an important role in the upcoming convergence schemes. Since $\dom f$ and $\dom g$ are bounded, $D_f$ and $D_g$ are real numbers.
\begin{proposition}
	\label{corollary:f*rho and g*mu inequality}
	For all $p\in\mathbb{R}^m$ it holds
	\begin{align*}
		f_{\rho}^*(A^*p) &\leq f^*(A^*p) \leq f_{\rho}^*(A^*p) + \rho D_f  \ \mbox{and} \ g_{\mu}^*(-p) \leq g^*(-p) \leq g_{\mu}^*(-p) + \mu D_g.
\end{align*}
\end{proposition}
\begin{proof} For $p \in \mathbb{R}^m$ one has
\begin{align*}
		f_{\rho}^*(A^*p) &= \left\langle  A^*p,x_{\rho,p} \right\rangle  -f(x_{\rho,p}) - \frac{\rho}{2} \left\| x_{\rho,p} \right\|^2
											\leq \left\langle  A^*p,x_{\rho,p} \right\rangle  -f(x_{\rho,p}) \leq  f^*(A^*p) \\
									&\leq \sup_{x \in \dom f}{ \left\{ \left\langle  A^*p,x \right\rangle  -f(x) -\frac{\rho}{2} \left\| x \right\|^2 \right\} } + \sup_{x \in \dom f}{ \left\{\frac{\rho}{2} \left\| x \right\|^2\right\} } \\
									&= f_{\rho}^*(A^*p) + \rho D_f.
\end{align*}
The other estimates follow similarly.
\end{proof}

For $\rho >0$ and $\mu >0$ let be $\theta_{\rho,\mu}:\mathbb{R}^m \rightarrow \mathbb{R}$ defined by $\theta_{\rho,\mu}(p)=f_{\rho}^*(A^*p) + g_{\mu}^*(-p)$. The function $\theta_{\rho,\mu}$ is differentiable with a Lipschitz continuous gradient
$$\nabla \theta_{\rho,\mu}(p) = \nabla(f_{\rho}^*\circ A^*)(p) + \nabla g_{\mu}^*(-\cdot)(p)=Ax_{\rho,p} - x_{\mu,p}$$
having as Lipschitz constant $L(\rho,\mu) := \frac{\left\| A \right\|^2}{\rho} + \frac{1}{\mu}$.

Summing up the inequalities from Proposition \ref{corollary:f*rho and g*mu inequality}, we get
\begin{equation}
		\label{inequality:theta-rho-mu and theta relation}
		\theta_{\rho,\mu}(p) \leq \theta (p) \leq \theta_{\rho,\mu}(p) + \rho D_f + \mu D_g \quad \forall p \in \mathbb{R}^m.
\end{equation}
Further, for $p \in \mathbb{R}^m$ we have
\begin{align*}
		\theta_{\rho,\mu}(p)&=	f_{\rho}^*(A^*p) + g_{\mu}^*(-p) \\
												&= \left\langle  p,Ax_{\rho,p} \right\rangle  -f(x_{\rho,p}) - \frac{\rho}{2} \left\| x_{\rho,p} \right\|^2
												 -\left\langle  p,x_{\mu,p} \right\rangle  -g(x_{\mu,p}) - \frac{\mu}{2} \left\| x_{\mu,p} \right\|^2
\end{align*}
and from here
\begin{align*}
		f(x_{\rho,p}) + g(x_{\mu,p}) -v(D) = \left\langle  p,\nabla \theta_{\rho,\mu}(p) \right\rangle + (-v(D) -\theta_{\rho,\mu}(p) )  - \frac{\rho}{2} \left\| x_{\rho,p} \right\|^2 - \frac{\mu}{2} \left\| x_{\mu,p} \right\|^2.
\end{align*}
Thus
\begin{equation}\label{ineq:primal-dual}
		\left| f(x_{\rho,p}) + g(x_{\mu,p})-v(D) \right|
		\leq 	\left| \left\langle  p,\nabla \theta_{\rho,\mu}(p) \right\rangle \right| + \left|v(D) + \theta_{\rho,\mu}(p) \right|
					+ \rho D_f + \mu D_g.
\end{equation}

Since $v(P) \geq v(D)$ (weak duality) and $\left| \theta_{\rho,\mu}(p) + v(D) \right| \overset{\eqref{inequality:theta-rho-mu and theta relation}}{\leq} \left| \theta(p) + v(D) \right| + \rho D_f + \mu D_g$, we conclude that
\begin{equation}
	\label{inequality:first-smoothing-primal-accuray}
	f(x_{\rho,p}) + g(x_{\mu,p}) - v(P)	\leq \left| \left\langle  p,\nabla \theta_{\rho,\mu}(p) \right\rangle \right|
																		 + \left|\theta(p) + v(D) \right| + 2\rho D_f + 2\mu D_g.
\end{equation}
Following the ideas in \cite{NesterovDoubleSmooth}, we further consider for the regularized optimization problem (for $\rho > 0$ and $\mu > 0$)
\begin{equation}
		\label{opt-problem:theta-rho-mu as objective function}
		\inf_{p\in\mathbb{R}^m}{\theta_{\rho,\mu}(p)}
\end{equation}
the following fast gradient scheme (see \cite[scheme (3.11)]{NesterovSmoothMin05}):
\begin{align*}
			\text{Init.: } &\text{Choose }w_0 \in \mathbb{R}^m \text{ and set }k:=0. \\
		\text{For } k\geq 0:\ &\text{Compute}\ \theta_{\rho,\mu}(w_k) \text{ and } \nabla \theta_{\rho,\mu}(w_k). \\
			 &\text{Find }p_k=\underset{w\in\mathbb{R}^m}{\argmin }{\left\{\left\langle \nabla \theta_{\rho,\mu}(w_k),w-w_k \right\rangle + \frac{L(\rho,\mu)}{2} \left\| w-w_k \right\|^2\right\}}. \\
			 &\text{Find }z_k=\underset{w\in\mathbb{R}^m}{\argmin } \bigg \{L(\rho,\mu)\left\|w_0-w\right\|^2 \\
            & \hspace{4cm} + \sum_{i=0}^{k}{\frac{i+1}{2}\left[ \theta_{\rho,\mu}(w_i) + \left\langle \nabla \theta_{\rho,\mu}(w_i),w-w_i \right\rangle\right]} \bigg \}. \\			&\text{Set }w_{k+1}:= \frac{2}{k+3}z_k  + \frac{k+1}{k+3}p_k.
\end{align*}
Assuming that $p_S^* \in \mathbb{R}^m$ is an \textit{optimal solution} of \eqref{opt-problem:theta-rho-mu as objective function}, it follows that
$\nabla \theta_{\rho,\mu}(p_S^*)=0$. Thus, due to the properties of the above convergence scheme provided in \cite{NesterovSmoothMin05}, we have
 \begin{equation}
		\label{inequality:pseudo-fast-gradient-scheme}
		\theta_{\rho,\mu}(p_k)-\theta_{\rho,\mu}(p_S^*) \leq \frac{4L(\rho,\mu) \left\| p_0 - p_S^* \right\|^2 }{(k+1)(k+2)} \quad \forall\,k\geq 0.
 \end{equation}
When $p^* \in \R^m$ is an \textit{optimal solution} to $(D)$, from \eqref{inequality:theta-rho-mu and theta relation} we get that $\theta_{\rho,\mu}(p_k) \geq \theta(p_k) - \rho D_f - \mu D_g$ for all $k\geq 0$ and $\theta_{\rho,\mu}(p_S^*) \leq \theta_{\rho,\mu}(p^*) \leq \theta(p^*) = -v(D)$. Hence, we obtain
\begin{align*}
			\theta_{\rho,\mu}(p_k) - \theta_{\rho,\mu}(p_S^*) \geq \theta(p_k) - \rho D_f  - \mu D_g +v(D),
\end{align*}
which further implies that
\begin{align*}
			\theta(p_k) + v(D) \leq \theta_{\rho,\mu}(p_k) - \theta_{\rho,\mu}(p_S^*) + \rho D_f  + \mu D_g
			\overset{\eqref{inequality:pseudo-fast-gradient-scheme}}{\leq}  \frac{4L(\rho,\mu) \left\| p_0 - p_S^* \right\|^2 }{(k+1)(k+2)} + \rho D_f  + \mu D_g
\end{align*}
for all $k\geq 0$. Now, in order to guarantee $\theta(p_k) + v(D) \leq \epsilon$, namely that $p_k$ is a solution of the dual problem $(D)$ with $\epsilon$-accuracy, we can force all three terms in the above inequality to be less than or equal to $\frac{\epsilon}{3}$. By taking
\begin{align*}
			\rho := \rho(\epsilon) = \frac{\epsilon}{3 D_f} \ \mbox{and} \ \mu := \mu(\epsilon) = \frac{\epsilon}{3 D_g},
\end{align*}
this means that the amount of iterations $k$ needed in order to satisfy $\epsilon$-optimality for the dual iterate depends on the relation
\begin{align*}
			\frac{4L(\rho,\mu) \left\| p_0 - p_S^* \right\|^2 }{(k+1)(k+2)} \leq \frac{\epsilon}{3}.
\end{align*}
Since the Lipschitz constant $L(\rho,\mu)=\frac{\left\| A \right\|^2}{\rho} + \frac{1}{\mu}$ is of order $\frac{1}{\epsilon}$, the rate of convergence for $\theta(p_k) +v(D) \leq \epsilon$ is $O\left(\frac{1}{\epsilon}\right)$.

 Further, according to \eqref{inequality:first-smoothing-primal-accuray}, in order to gain an accuracy for the primal optimization problem  proportional to $\epsilon>0$, one has only to ensure that $\left| \left\langle  p_k,\nabla \theta_{\rho,\mu}(p_k) \right\rangle \right|$ is lower than or equal to $O(\epsilon)$. However, by \cite[Theorem 2.1.5]{Nesterov04}, we have
 \begin{equation*}
		\left\| \nabla \theta_{\rho,\mu}(p_k) \right\|^2 \leq 2 L(\rho,\mu)(\theta_{\rho,\mu}(p_k)-\theta_{\rho,\mu}(p_S^*)),
 \end{equation*}
hence, from \eqref{inequality:pseudo-fast-gradient-scheme},
 \begin{equation*}
		\left\| \nabla \theta_{\rho,\mu}(p_k) \right\| \leq \frac{2\sqrt{2}L(\rho,\mu) \left\| p_0 - p_S^* \right\| }{\sqrt{\left(k+1\right) \left(k+2\right) }}\quad \forall\,k\geq0.
 \end{equation*}
This means that the norm of the gradient $\nabla \theta_{\rho,\mu}(p_k)$ decreases with an order being $O\left(\frac{1}{\epsilon^2}\right)$. In order to achieve for the primal optimization problem an accuracy  which is proportional to $\epsilon$ via the estimation  \eqref{inequality:first-smoothing-primal-accuray}, we need $k=O\left(\frac{1}{\epsilon^2}\right)$ iterations.   This convergence is slow as compared to our aimed rate of convergence of $O\left(\frac{1}{\epsilon}\ln{\left(\frac{1}{\epsilon}\right)}\right)$ and it is not better than the rate of convergence of the subgradient approach.

From another point of view, in order to get a feasible solution to the primal optimization problem $(P)$, it is necessary to investigate the distance between $Ax_{\rho,p_k}$ and $x_{\mu,p_k}$, since the functions $f$ and $g \circ A$ have to share the same argument (which would be $x_{\rho,p_k}$, if $\left\| \nabla \theta_{\rho,\mu}(p_k) \right\| = \left\| Ax_{\rho,p_k} - x_{\mu,p_k} \right\| = 0$). Therefore, the norm of the gradient $\left\| \nabla \theta_{\rho,\mu}(p_k) \right\|$ is an indicator for an approximately feasible solution. Thus, in order to obtain an approximately optimal solution to $(P)$, it is not sufficient to ensure the convergence for $\theta(p_k) + v(D)$ to zero, but also a good convergence for the decrease of $\left\| \nabla \theta_{\rho,\mu}(p_k) \right\|$.

\subsection{Second smoothing}\label{subsectionSecondSmoothing}

In the following a second regularization is applied to $\theta_{\rho,\mu}$, as done in \cite{NesterovDoubleSmooth, NesterovDoubleSmoothInfinite}, in order to make it strongly convex, fact which will allow us to use a fast gradient scheme with a better convergence rate for $\left\| \nabla \theta_{\rho,\mu} \right\|$. Therefore, adding the strongly convex function $\frac{\kappa}{2} \left\| \cdot \right\|^2$ to $\theta_{\rho,\mu}$ for some positive real number $\kappa$ gives rise to the following regularization of the objective function
\begin{equation*}
\theta_{\rho,\mu,\kappa} : \R^m \rightarrow \R, \ \theta_{\rho,\mu,\kappa}(p) := \theta_{\rho,\mu}(p) + \frac{\kappa}{2} \left\| p \right\|^2
																= f_{\rho}^*(A^*p) + g_{\mu}^*(-p)+ \frac{\kappa}{2} \left\| p \right\|^2,
\end{equation*}
which is strongly convex with modulus $\kappa > 0$ (cf. \cite[Proposition B.1.1.2]{HirLem01}). We further deal with the optimization problem
\begin{equation}
		\label{opt-problem:second-regularization}
		\inf_{p \in \mathbb{R}^m}{\theta_{\rho,\mu,\kappa}(p)}.
\end{equation}
By taking into account \cite[Proposition A.8 and Proposition B.10]{Bertsekas99}, the optimization problem \eqref{opt-problem:second-regularization} has an unique element. The function $\theta_{\rho,\mu,\kappa}$ is differentiable and for all $p \in \R^m$ it holds
\begin{align*}
	\nabla \theta_{\rho,\mu,\kappa}(p) 	= \nabla \left( \theta_{\rho,\mu} (\cdot)+ \frac{\kappa}{2} \left\| \cdot \right\|^2 \right) (p)
																			= Ax_{\rho,p} - x_{\mu,p} + \kappa p.
\end{align*}
This gradient is Lipschitz continuous with constant $L(\rho,\mu,\kappa):=\frac{\left\| A \right\|^2}{\rho} + \frac{1}{\mu} + \kappa$.

\section{Solving the doubly regularized dual problem}\label{sectionConvergence}

\subsection{An appropriate fast gradient method}

Denote by $p_{DS}^*$ the unique optimal solution to optimization problem \eqref{opt-problem:second-regularization} and by $\theta_{\rho,\mu,\kappa}^* := \theta_{\rho,\mu,\kappa}(p_{DS}^*)$ its optimal objective value. Further, let $p^* \in \mathbb{R}^m$ be an optimal solution to the dual optimization problem $(D)$ and assume that the upper bound
\begin{equation}
		\label{opt-problem:dual-solution-upper-bound}
		\left\| p^* \right\| \leq R																					
\end{equation}
is available for some nonzero $R \in \mathbb{R}_+$.

We apply to the doubly regularized dual problem \eqref{opt-problem:second-regularization}
the fast gradient method \cite[Algorithm 2.2.11]{Nesterov04}
\begin{eqnarray}
		\label{doublesmooth:opt-sheme}
		\text{Init.:}& &\text{Set}\ w_0 = p_0:=0 \in \mathbb{R}^m \notag \\
		\text{For } k\geq 0:& &\text{Set}\ p_{k+1} := w_k - \frac{1}{L(\rho,\mu,\kappa)}\nabla \theta_{\rho,\mu,\kappa}(w_k).\\
			& &\text{Set}\ w_{k+1} := p_{k+1} + \frac{\sqrt{L(\rho,\mu,\kappa)} - \sqrt{\kappa}}{\sqrt{L(\rho,\mu,\kappa)} + \sqrt{\kappa}} (p_{k+1}-p_k).\notag
\end{eqnarray}
By taking into account \cite[Theorem 2.2.3]{Nesterov04} we obtain a sequence $(p_k)_{k \geq 0} \subseteq \mathbb{R}^m$ satisfying
\begin{align}
		\theta_{\rho,\mu,\kappa}(p_k) - \theta_{\rho,\mu,\kappa}^*
		&\leq \left( \theta_{\rho,\mu,\kappa}(p_0) - \theta_{\rho,\mu,\kappa}^* + \frac{\kappa}{2} \left\| p_0 -p_{DS}^* \right\|^2\right)
					\left( 1-\sqrt{\frac{\kappa}{L(\rho,\mu,\kappa)}} \right)^k \notag \\
		\label{opt-scheme:objective-function1}
		&\leq  (\theta_{\rho,\mu,\kappa}(p_0) - \theta_{\rho,\mu,\kappa}^* + \frac{\kappa}{2} \left\| p_0 -p_{DS}^* \right\|^2)
																																		 \text{e}^{ -k \sqrt{\frac{\kappa}{L(\rho,\mu,\kappa)}} } \\
		\label{opt-scheme:objective-function2}																																 
		&\leq  2 (\theta_{\rho,\mu,\kappa}(p_0) - \theta_{\rho,\mu,\kappa}^*)
																																		 \text{e}^{ -k \sqrt{\frac{\kappa}{L(\rho,\mu,\kappa)}} } \ \forall k \geq 0,
\end{align}
while the last inequality is a consequence of \cite[Theorem 2.1.8]{Nesterov04}. Since $p_{DS}^*$ is the unique optimal solution to \eqref{opt-problem:second-regularization}, we have $\nabla \theta_{\rho,\mu,\kappa}(p_{DS}^*)=0$ and therefore \cite[Theorem 2.1.5]{Nesterov04} yields
\begin{align*}
		\frac{1}{2 L(\rho,\mu,\kappa)}\left\| \nabla \theta_{\rho,\mu,\kappa}(p_k) \right\|^2
		\leq 	\theta_{\rho,\mu,\kappa}(p_k) - \theta_{\rho,\mu,\kappa}^*
		\overset{\eqref{opt-scheme:objective-function2}}{\leq} 2 (\theta_{\rho,\mu,\kappa}(p_0) - \theta_{\rho,\mu,\kappa}^*)
\text{e}^{ -k \sqrt{\frac{\kappa}{L(\rho,\mu,\kappa)}}},																																 
\end{align*}
which implies
\begin{align}
		\label{opt-scheme:norm-of-gradient}
		\left\| \nabla \theta_{\rho,\mu,\kappa}(p_k) \right\|^2
		\leq 	4 L(\rho,\mu,\kappa) (\theta_{\rho,\mu,\kappa}(p_0) - \theta_{\rho,\mu,\kappa}^*)
					\text{e}^{ -k \sqrt{\frac{\kappa}{L(\rho,\mu,\kappa)}} }  \ \forall k \geq 0.																																 
\end{align}
Due to the strong convexity of $\theta_{\rho,\mu,\kappa}$ with modulus $\kappa > 0$, Theorem 2.1.8 in \cite{Nesterov04} states
\begin{align}
	\label{opt-scheme:norm-of-arguments-additional}
	\frac{\kappa}{2} \left\| p_k -p_{DS}^* \right\|^2
	\leq \theta_{\rho,\mu,\kappa}(p_k) - \theta_{\rho,\mu,\kappa}^*
	\overset{\eqref{opt-scheme:objective-function2}}{\leq} 2 (\theta_{\rho,\mu,\kappa}(p_0) - \theta_{\rho,\mu,\kappa}^*)
\text{e}^{ -k\sqrt{\frac{\kappa}{L(\rho,\mu,\kappa)}} }  \ \forall k \geq 0.
\end{align}
Using this inequality it follows that (see also \cite{NesterovDoubleSmooth,NesterovDoubleSmoothInfinite})
\begin{equation}
		\label{opt-scheme:norm-of-arguments}
		\left\| p_k - p_{DS}^* \right\|^2
		\leq 	\min{ \left\{ \left\| p_0 - p_{DS}^* \right\|^2, \frac{4}{\kappa} (\theta_{\rho,\mu,\kappa}(p_0) - \theta_{\rho,\mu,\kappa}^*)
					\text{e}^{ -k \sqrt{\frac{\kappa}{L(\rho,\mu,\kappa)}} }\right\} }	\ \forall k \geq 0.																				 
\end{equation}
We will show as follows that the rates of convergence for the decrease of $\left\| \nabla \theta_{\rho,\mu}(p_k) \right\|$ and $\theta(p_k) + v(D)$ are the same, namely equal to $O\left(\frac{1}{\epsilon} \ln \left(\frac{1}{\epsilon}\right)\right)$. This will us allow to efficiently recover approximately optimal solutions to the initial optimization problem $(P)$.

\subsection{Convergence of \texorpdfstring{$\theta(p_k)$}{theta(pk)} to \texorpdfstring{$-v(D)$}{-v(D)}}\label{subsectionConvTheta}

Since $p_0 = 0$, we have
$$\theta_{\rho,\mu,\kappa}(0)=f_{\rho}^*(0)+g_{\mu}^*(0)+\frac{\kappa}{2} \left\| 0 \right\|^2 = \theta_{\rho,\mu}(0)$$
and
\begin{align}
	\label{inequality:theta-rho-mu-kappa in p_{DS}^*}
	\theta_{\rho,\mu,\kappa}(p_{DS}^*) = \theta_{\rho,\mu}(p_{DS}^*) + \frac{\kappa}{2} \left\| p_{DS}^* \right\|^2
\end{align}
and obtain
\begin{align*}
		\frac{\kappa}{2} \left\| p_{DS}^* \right\|^2 &\overset{\eqref{opt-scheme:norm-of-arguments-additional}}{\leq} \theta_{\rho,\mu,\kappa}(0) - \theta_{\rho,\mu,\kappa}(p_{DS}^*)																									 = \theta_{\rho,\mu}(0) - \theta_{\rho,\mu}(p_{DS}^*) - \frac{\kappa}{2} \left\| p_{DS}^* \right\|^2,
\end{align*}
which implies that
\begin{align}
		\label{inequality:norm-of-p_DS^*}
		\left\| p_{DS}^* \right\|^2 &\leq \frac{1}{\kappa} \left( \theta_{\rho,\mu}(0) - \theta_{\rho,\mu}(p_{DS}^*) \right).		
\end{align}
In addition, for all $k\geq 0$ it holds
\begin{align}
		\left\| p_k - p_{DS}^* \right\|^2
		&\overset{\eqref{opt-scheme:norm-of-arguments-additional}}{\leq} \frac{2}{\kappa} \left( \theta_{\rho,\mu,\kappa}(p_k) - \theta_{\rho,\mu,\kappa}(p_{DS}^*) \right) \notag \\
		&\overset{\eqref{opt-scheme:objective-function1}}{\leq} \frac{2}{\kappa} \left( \theta_{\rho,\mu,\kappa}(0) - \theta_{\rho,\mu,\kappa}(p_{DS}^*) + \frac{\kappa}{2} \left\| 0 - p_{DS}^* \right\|^2 \right) \text{e}^{ -k \sqrt{\frac{\kappa}{L(\rho,\mu,\kappa)}} } \notag \\
		&\overset{\eqref{inequality:theta-rho-mu-kappa in p_{DS}^*}}{=} \frac{2}{\kappa} \left( \theta_{\rho,\mu}(0) - \theta_{\rho,\mu}(p_{DS}^*) \right) \text{e}^{ -k \sqrt{\frac{\kappa}{L(\rho,\mu,\kappa)}} }
		\label{inequality:norm-of-p_k-minus-p_DS^*}
\end{align}
and
\begin{align}
	\label{inequality:theta-rho-mu estimate}
		\theta_{\rho,\mu}(p_k) - \theta_{\rho,\mu}(p_{DS}^*)
		&\overset{\eqref{opt-scheme:objective-function1}}{\leq} \left( \theta_{\rho,\mu,\kappa}(0) - \theta_{\rho,\mu,\kappa}(p_{DS}^*) + \frac{\kappa}{2} \left\| 0 - p_{DS}^* \right\|^2 \right)
						\text{e}^{ -k \sqrt{\frac{\kappa}{L(\rho,\mu,\kappa)}} } \notag \\
						& \hspace{0.8cm}+ \frac{\kappa}{2} \left( \left\| p_{DS}^* \right\|^2
						- \left\| p_k \right\|^2	\right) \notag \\
		&\overset{\eqref{inequality:theta-rho-mu-kappa in p_{DS}^*}}{=} \left( \theta_{\rho,\mu}(0) - \theta_{\rho,\mu}(p_{DS}^*)  \right)
						\text{e}^{ -k \sqrt{\frac{\kappa}{L(\rho,\mu,\kappa)}} } + \frac{\kappa}{2} \left( \left\| p_{DS}^* \right\|^2
						- \left\| p_k \right\|^2	\right).																					 
\end{align}
Investigating the last term in the estimate above, using $\left| \left\| p_{DS}^* \right\| - \left\| p_k \right\| \right| \leq \left\| p_{DS}^* - p_k \right\|$ and $\left\| p_k \right\| = \left\| p_k -p_{DS}^* + p_{DS}^* \right\| \leq \left\| p_k - p_{DS}^* \right\| + \left\| p_{DS}^* \right\|$, we get for all $k\geq 0$
\begin{eqnarray*}
		\left\| p_{DS}^* \right\|^2 - \left\| p_k \right\|^2
		&=& \left(\left\| p_{DS}^* \right\| - \left\| p_k \right\|\right)\left(\left\| p_{DS}^* \right\| + \left\| p_k \right\|\right) \\
		&\leq& \left\| p_{DS}^* - p_k \right\| \left( \left\| p_{DS}^* \right\| + \left\| p_k \right\| \right) \\
		&\leq& \left\| p_{DS}^* - p_k \right\| \left( 2\left\| p_{DS}^* \right\| + \left\| p_k - p_{DS}^* \right\| \right) \\
		&\overset{\eqref{opt-scheme:norm-of-arguments}}{\leq}& 3 \left\| p_{DS}^* - p_k \right\| \left\| p_{DS}^*\right\| \\
		&\overset{\eqref{inequality:norm-of-p_k-minus-p_DS^*}}{\leq}& 3 \left\| p_{DS}^* \right\|
							\sqrt{\frac{2}{\kappa} (\theta_{\rho,\mu}(0) - \theta_{\rho,\mu}(p_{DS}^*))} \, \text{e}^{ -\frac{k}{2} \sqrt{\frac{\kappa}{L(\rho,\mu,\kappa)}} } \\
		&\overset{\eqref{inequality:norm-of-p_DS^*}}{\leq}& 	\frac{3 \sqrt{2}}{\kappa} (\theta_{\rho,\mu}(0) - \theta_{\rho,\mu}(p_{DS}^*)) \, \text{e}^{ -\frac{k}{2} \sqrt{\frac{\kappa}{L(\rho,\mu,\kappa)}} }	.																						 
\end{eqnarray*}
Inserting this result into \eqref{inequality:theta-rho-mu estimate}, we obtain for all $k\geq 0$
\begin{align}
		\label{inequality:theta-rho-mu estimate2}
		\theta_{\rho,\mu}(p_k) - \theta_{\rho,\mu}(p_{DS}^*)
		&\leq \left( \theta_{\rho,\mu}(0) - \theta_{\rho,\mu}(p_{DS}^*)  \right)
						\left( \text{e}^{ -k \sqrt{\frac{\kappa}{L(\rho,\mu,\kappa)}}} + \frac{3}{\sqrt{2}} \, \text{e}^{ -\frac{k}{2} \sqrt{\frac{\kappa}{L(\rho,\mu,\kappa)}}} \right) \notag \\
		&\leq \frac{25}{8} \left( \theta_{\rho,\mu}(0) - \theta_{\rho,\mu}(p_{DS}^*)  \right) \text{e}^{ -\frac{k}{2} \sqrt{\frac{\kappa}{L(\rho,\mu,\kappa)}}}.																			
\end{align}
Further, we have $\theta_{\rho,\mu}(0) \overset{\eqref{inequality:theta-rho-mu and theta relation}}{\leq} \theta(0)$ and
\begin{align*}
		\theta_{\rho,\mu}(p_{DS}^*) \overset{\eqref{inequality:theta-rho-mu and theta relation}}{\geq} \theta(p_{DS}^*) - \rho D_f - \mu D_g
		\geq \theta(p^*) - \rho D_f - \mu D_g,
\end{align*}
and, from here,
\begin{align}
		\label{inequality:theta-0-p^*-estimate}
		\theta_{\rho,\mu}(0) - \theta_{\rho,\mu}(p_{DS}^*)
		\leq \theta(0) - \theta(p^*) + \rho D_f + \mu D_g.
\end{align}
Finally, since $\theta_{\rho,\mu}(p_{DS}^*) \leq \theta_{\rho,\mu}(p_{DS}^*) + \frac{\kappa}{2} \left\| p_{DS}^* \right\|^2 \leq \theta_{\rho,\mu}(p^*) + \frac{\kappa}{2} \left\| p^* \right\|^2 $, we conclude that
\begin{align*}
		\theta_{\rho,\mu}(p_{DS}^*) \leq \theta_{\rho,\mu}(p^*) + \frac{\kappa}{2} \left\| p^* \right\|^2
		\overset{\eqref{inequality:theta-rho-mu and theta relation}}{\leq} \theta(p^*) + \frac{\kappa}{2} \left\| p^* \right\|^2
\end{align*}
and, therefore, for all $k\geq 0$
\begin{align}
		\label{inequality:theta-estimate1}
		\theta_{\rho,\mu}(p_k) - \theta_{\rho,\mu}(p_{DS}^*)
		\overset{\eqref{inequality:theta-rho-mu and theta relation}}{\geq} \theta(p_k) - \rho D_f - \mu D_g - \theta(p^*) - \frac{\kappa}{2} \left\| p^* \right\|^2.
\end{align}
In conclusion, we obtain for all $k\geq 0$
\begin{eqnarray}
		\label{inequality:theta-estimate2}
		\theta(p_k) - \theta(p^*) &\overset{\eqref{inequality:theta-estimate1}}{\leq}&
															\rho D_f + \mu D_g + \frac{\kappa}{2} \left\| p^* \right\|^2 + \theta_{\rho,\mu}(p_k) - \theta_{\rho,\mu}(p_{DS}^*) \notag \\
		&\overset{\eqref{opt-problem:dual-solution-upper-bound}, \eqref{inequality:theta-rho-mu estimate2}}{\leq}&
				\rho D_f + \mu D_g + \frac{\kappa}{2} R^2 + \frac{25}{8} \left( \theta_{\rho,\mu}(0) - \theta_{\rho,\mu}(p_{DS}^*)  \right) \text{e}^{ -\frac{k}{2} \sqrt{\frac{\kappa}{L(\rho,\mu,\kappa)}}} \notag \\
		&\overset{\eqref{inequality:theta-0-p^*-estimate}}{\leq}&
				\rho D_f + \mu D_g + \frac{\kappa}{2} R^2 \notag \\ & &+ \frac{25}{8} \left( \theta(0) - \theta(p^*) + \rho D_f + \mu D_g  \right) \text{e}^{ -\frac{k}{2} \sqrt{\frac{\kappa}{L(\rho,\mu,\kappa)}}}.
\end{eqnarray}
Next we fix $\epsilon > 0$. In order to get $\theta(p_k) + v(D) \leq \epsilon$ for a certain amount of iterations $k$, we force all four terms in \eqref{inequality:theta-estimate2} to be less than or equal to $\frac{\epsilon}{4}$. Therefore, we choose
\begin{align}
	\label{opt-problem-smoothing-parameters}
	\rho := \rho(\epsilon) = \frac{\epsilon}{4 D_f}, \quad
	\mu := \mu(\epsilon) = \frac{\epsilon}{4 D_g}, \quad
	\kappa := \kappa(\epsilon) = \frac{\epsilon}{2 R^2}.
\end{align}
With these new parameters we can simplify \eqref{inequality:theta-estimate2} to
\begin{equation*}
		\theta(p_k) + v(D) \leq
		\frac{3 \epsilon}{4} + \frac{25}{8} \left( \theta(0) - \theta(p^*) + \frac{\epsilon}{2}  \right) \text{e}^{ -\frac{k}{2} \sqrt{\frac{\kappa}{L(\rho,\mu,\kappa)}}}.
\end{equation*}
As we see, the second term in the expression on the right-hand side of the above estimate determines the number of iterations which is needed to obtain $\epsilon$-accuracy for the dual objective function $\theta$. Indeed, we have
\begin{align}
		\frac{\epsilon}{4} &\geq \frac{25}{8} \left( \theta(0) - \theta(p^*) + \frac{\epsilon}{2}  \right) \text{e}^{ -\frac{k}{2} \sqrt{\frac{\kappa}{L(\rho,\mu,\kappa)}}} \notag \\
		\Leftrightarrow \text{e}^{ \frac{k}{2} \sqrt{\frac{\kappa}{L(\rho,\mu,\kappa)}}} &\geq  \frac{4}{\epsilon} \cdot \frac{25}{8} \left( \theta(0) - \theta(p^*) + \frac{\epsilon}{2}  \right) \notag \\
		\Leftrightarrow \frac{k}{2} \sqrt{\frac{\kappa}{L(\rho,\mu,\kappa)}} &\geq \ln{\left( \frac{25\left( \theta(0) - \theta(p^*) + \frac{\epsilon}{2}  \right)}{2 \epsilon}  \right)} \notag \\
		\Leftrightarrow k &\geq 2 \sqrt{\frac{L(\rho,\mu,\kappa)}{\kappa}} \ln{\left( \frac{25\left( \theta(0) - \theta(p^*) + \frac{\epsilon}{2}  \right)}{2 \epsilon}  \right)}  \label{doublesmooth-estimation-theta}
\end{align}
iterations. A closer look on $\frac{L(\rho,\mu,\kappa)}{\kappa}$ shows that
\begin{eqnarray*}
		\frac{L(\rho,\mu,\kappa)}{\kappa}
		= \frac{\left\| A \right\|^2}{\rho \kappa} + \frac{1}{\mu \kappa} + 1
		&\overset{\eqref{opt-problem-smoothing-parameters}}{=}&\frac{8 \left\| A \right\|^2 D_f R^2}{\epsilon^2} + \frac{8 D_g R^2}{\epsilon^2} + 1 \\
		&=& 1 + \frac{8 R^2}{\epsilon^2} \left( \left\| A \right\|^2 D_f + D_g \right),
\end{eqnarray*}
hence, in order to obtain an approximately optimal solution to $(D)$, we need $k=O\left(\frac{1}{\epsilon}\ln \left(\frac{1}{\epsilon}\right)\right)$ iterations.

\subsection{Convergence of \texorpdfstring{$\left\| \nabla\theta_{\rho,\mu}(p_k) \right\|$}{norm nabla-theta-rho-mu} to \texorpdfstring{0}{0}}\label{subsectionConvNablaTheta}

As it follows from \eqref{inequality:first-smoothing-primal-accuray}, guaranteeing $\epsilon$-optimality for the objective values of $\theta$ is not sufficient for solving the initial primal optimization problem with a good convergence rate in the absence of a similar behavior of $\|\nabla \theta_{\rho,\mu}(p_k)\| = \left\| Ax_{\rho,p_k} - x_{\mu,p_k} \right\|$. In the following we show that the fast gradient method \eqref{doublesmooth:opt-sheme} applied to the doubly regularized function $\theta_{\rho,\mu,\kappa}$ furnishes the desired properties for the decrease of $\left\| \nabla\theta_{\rho,\mu}(p_k) \right\|$ (see also \cite{NesterovDoubleSmooth, NesterovDoubleSmoothInfinite}).
Since
\begin{align*}
	\left\| p_k \right\| = \left\| p_k - p_{DS}^* +p_{DS}^* \right\| \leq \left\| p_k - p_{DS}^* \right\| + \left\| p_{DS}^* \right\|
		\overset{\eqref{opt-scheme:norm-of-arguments}}{\leq} 2 \left\| p_{DS}^* \right\| ,
\end{align*}
we have
\begin{align}
	\label{inequality:norm-of-theta-rho-mu}
	\left\| \nabla \theta_{\rho,\mu}(p_k) \right\| &= \left\| \nabla\theta_{\rho,\mu}(p_k) + \kappa p_k  - \kappa  p_k  \right\| = \left\| \nabla\theta_{\rho,\mu,\kappa}(p_k) - \kappa  p_k  \right\| \notag \\
	&\leq \left\|\nabla \theta_{\rho,\mu,\kappa}(p_k) \right\| + \left\| \kappa p_k \right\|
		= \left\| \nabla\theta_{\rho,\mu,\kappa}(p_k) \right\| + \kappa \left\| p_k \right\| \notag \\
		&\leq \left\| \nabla\theta_{\rho,\mu,\kappa}(p_k) \right\| + 2\kappa \left\| p_{DS}^* \right\|	\ \forall k \geq 0.
\end{align}
Having a closer look on the first term in the previous estimate one can notice that
\begin{align*}
	\left\| \nabla\theta_{\rho,\mu,\kappa}(p_k) \right\|^2 &\overset{\eqref{opt-scheme:norm-of-gradient}}{\leq}
					4 L(\rho,\mu,\kappa) (\theta_{\rho,\mu,\kappa}(0) - \theta_{\rho,\mu,\kappa}(p_{DS}^*)) \, \text{e}^{ -k \sqrt{\frac{\kappa}{L(\rho,\mu,\kappa)}} } \\
					&\overset{\eqref{inequality:theta-rho-mu-kappa in p_{DS}^*}}{\leq}
					4 L(\rho,\mu,\kappa) (\theta_{\rho,\mu}(0) - \theta_{\rho,\mu}(p_{DS}^*))\, \text{e}^{ -k \sqrt{\frac{\kappa}{L(\rho,\mu,\kappa)}} } \\
					&\overset{\eqref{opt-problem-smoothing-parameters}}{=}
					4 L(\rho,\mu,\kappa) \left(\theta(0) - \theta(p^*) + \frac{\epsilon}{2}\right)\, \text{e}^{ -k \sqrt{\frac{\kappa}{L(\rho,\mu,\kappa)}} },
\end{align*}
thus,
\begin{align}
	\label{inequality:norm-of-theta-rho-mu-kappa}
	\left\| \nabla\theta_{\rho,\mu,\kappa}(p_k) \right\| \leq
					2 \sqrt{L(\rho,\mu,\kappa) \left(\theta(0) - \theta(p^*) + \frac{\epsilon}{2}\right)}\, \text{e}^{ -\frac{k}{2} \sqrt{\frac{\kappa}{L(\rho,\mu,\kappa)}} }\ \forall k \geq 0.
\end{align}
Furthermore, in order to gain an upper bound for the norm of $p^*_{DS}$, we notice that
\begin{eqnarray*}
	\theta(p^*) + \frac{\kappa}{2} \left\| p^* \right\|^2
	&\overset{\eqref{inequality:theta-rho-mu and theta relation}}{\geq}& \theta_{\rho,\mu}(p^*) + \frac{\kappa}{2} \left\| p^* \right\|^2
	\geq \theta_{\rho,\mu}(p_{DS}^*) + \frac{\kappa}{2} \left\| p_{DS}^* \right\|^2 \\
	&\overset{\eqref{inequality:theta-rho-mu and theta relation}}{\geq}&
	\theta(p_{DS}^*) -\rho D_f -\mu D_g  + \frac{\kappa}{2} \left\| p_{DS}^* \right\|^2 \\
	&\geq&
	\theta(p^*) -\rho D_f -\mu D_g  + \frac{\kappa}{2} \left\| p_{DS}^* \right\|^2,
\end{eqnarray*}
which implies $\frac{\kappa}{2} \left\| p_{DS}^* \right\|^2 \leq \frac{\kappa}{2} \left\| p^* \right\|^2 + \rho D_f + \mu D_g$ or, equivalently,
\begin{align*}
	\left\| p_{DS}^* \right\|^2 \leq \left\| p^* \right\|^2 + \frac{2\rho}{\kappa}  D_f + \frac{2\mu}{\kappa}  D_g.
\end{align*}
Hence,
\begin{align}
	\label{inequality:norm of optimal solution p_DS^*}
	\left\| p_{DS}^* \right\|   \leq \sqrt{\left\| p^* \right\|^2 + \frac{2\rho}{\kappa}D_f + \frac{2\mu}{\kappa}D_g}
	\overset{\eqref{opt-problem-smoothing-parameters}}{=} \sqrt{\left\| p^* \right\|^2 + \frac{\epsilon}{2\kappa} + \frac{\epsilon}{2\kappa}}
	&\overset{\eqref{opt-problem-smoothing-parameters}}{=} \sqrt{\left\| p^* \right\|^2 + 2R^2} \notag \\
	&\overset{\eqref{opt-problem:dual-solution-upper-bound}}{\leq} \sqrt{3} R,
\end{align}
which, combined with \eqref{inequality:norm-of-theta-rho-mu} and \eqref{inequality:norm-of-theta-rho-mu-kappa}, provides the following estimate for the norm of the gradient of $\theta_{\rho,\mu}(p_k)$ for $k \geq 0$
\begin{align}
	\label{inequality:norm-final-theta-rho-mu}
	\left\| \nabla\theta_{\rho,\mu}(p_k) \right\| &\leq
	2 \sqrt{L(\rho,\mu,\kappa) \left(\theta(0) - \theta(p^*) + \frac{\epsilon}{2}\right)} \text{e}^{ -\frac{k}{2} \sqrt{\frac{\kappa}{L(\rho,\mu,\kappa)}} }
	+ 2\sqrt{3} \kappa R \notag \\
	&= 2 \sqrt{L(\rho,\mu,\kappa) \left(\theta(0) - \theta(p^*) + \frac{\epsilon}{2}\right)} \text{e}^{ -\frac{k}{2} \sqrt{\frac{\kappa}{L(\rho,\mu,\kappa)}} }	+ \frac{\sqrt{3}\epsilon}{R}.
\end{align}
For  $\epsilon > 0$ fixed, the first term in \eqref{inequality:norm-final-theta-rho-mu} decreases by the iteration counter $k$, while, in order to ensure that $\left\| \nabla\theta_{\rho,\mu}(p_k) \right\| \leq \frac{2 \epsilon}{R}$, we have to pass
\begin{align}
	\frac{2 \epsilon}{R} &\geq 2 \sqrt{L(\rho,\mu,\kappa) (\theta(0) - \theta(p^*) + \frac{\epsilon}{2})} \text{e}^{ -\frac{k}{2} \sqrt{\frac{\kappa}{L(\rho,\mu,\kappa)}} }	+ \frac{\sqrt{3}\epsilon}{R} \notag \\
 \Leftrightarrow	 \frac{(2-\sqrt{3}) \epsilon}{R} &\geq 2 \sqrt{L(\rho,\mu,\kappa) (\theta(0) - \theta(p^*) + \frac{\epsilon}{2})} \text{e}^{ -\frac{k}{2} \sqrt{\frac{\kappa}{L(\rho,\mu,\kappa)}}}	\notag \\
 \Leftrightarrow	 \text{e}^{ \frac{k}{2} \sqrt{\frac{\kappa}{L(\rho,\mu,\kappa)}}} &\geq \frac{2R\sqrt{L(\rho,\mu,\kappa) (\theta(0) - \theta(p^*) + \frac{\epsilon}{2})}}{(2-\sqrt{3}) \epsilon}  	\notag \\
	\Leftrightarrow	 \frac{k}{2} \sqrt{\frac{\kappa}{L(\rho,\mu,\kappa)}} &\geq \ln{ \left( \frac{\sqrt{4R^2L(\rho,\mu,\kappa) (\theta(0) - \theta(p^*) + \frac{\epsilon}{2})}}{(2-\sqrt{3}) \epsilon} \right) }  \notag \\
	\Leftrightarrow	 k  &\geq 2 \sqrt{\frac{L(\rho,\mu,\kappa)}{\kappa}} \ln{ \left( \frac{\sqrt{4R^2L(\rho,\mu,\kappa) (\theta(0) - \theta(p^*) + \frac{\epsilon}{2})}}{(2-\sqrt{3}) \epsilon} \right) } \notag \\
	\Leftrightarrow	 k  &\geq \frac{2}{\epsilon} \sqrt{\epsilon^2 + 8 R^2(\left\| A \right\|^2 D_f+D_g)} \notag \\
	& \ \ \ \cdot \ln{ \left( \frac{\sqrt{(2\epsilon^2+ 16R^2(\left\| A \right\|^2 D_f+D_g)) (\theta(0) - \theta(p^*) + \frac{\epsilon}{2})}}{(2-\sqrt{3}) \epsilon^{\frac{3}{2}}} \right) } \notag \\
	\Leftrightarrow	 k  &\geq \frac{3}{\epsilon} \sqrt{\epsilon^2 + 8 R^2( \left\| A \right\|^2 D_f+D_g)} \notag \\
	& \ \ \ \cdot \ln{ \left( \frac{\sqrt[3]{(2\epsilon^2+ 16R^2( \left\| A \right\|^2 D_f+D_g)) (\theta(0) - \theta(p^*) + \frac{\epsilon}{2})}}{(2-\sqrt{3})^{\frac{2}{3}} \epsilon} \right) }  \label{doublesmooth-estimation-norm-theta}
\end{align}
iterations of the fast gradient method  \eqref{doublesmooth:opt-sheme}. In the above estimate, we used that $\frac{L(\rho,\mu,\kappa)}{\kappa}= 1 + \frac{8R^2}{\epsilon^2}(\left\| A \right\|^2 D_f + D_g)$ and $L(\rho,\mu,\kappa)=\frac{4\left\| A \right\|^2 D_f}{\epsilon} + \frac{4D_g}{\epsilon} + \frac{\epsilon}{2R^2}$ (see \eqref{opt-problem-smoothing-parameters}).
Resuming the achievements in the last two subsections, it follows that $k=O\left(\frac{1}{\epsilon} \ln\left(\frac{1}{\epsilon}\right)\right)$ iterations are needed to guarantee
\begin{equation}
	\label{opt-scheme:convergence-for-dual-problem}
	\theta(p_k) + v(D) \leq \epsilon \ \mbox{and} \ \left\| \nabla \theta_{\rho,\mu}(p_k) \right\| \leq \frac{2\epsilon}{R}
\end{equation}
with a rate of convergence which is very similar except for constant factors.

\subsection{How to construct an approximately primal optimal solution}\label{subsectionApproximatePrimal}

Next, by making use of the approximate dual solution $p_k$, for $k \geq 0$, we construct an \textit{approximately primal optimal solution}  for the initial problem $(P)$ and investigate its accuracy. To this end we will make use of the sequences $(x_{\rho,p_k})_{k\geq0} \subseteq \dom f$ and $(x_{\mu,p_k})_{k\geq0} \subseteq \dom g$ which are delivered by the algorithmic scheme \eqref{doublesmooth:opt-sheme}. We will prove that, given a fixed accuracy $\epsilon > 0$, we are able to reconstruct an approximately primal optimal solution such that, for $\rho$ and $\mu$ chosen as in \eqref{opt-problem-smoothing-parameters}, one gets
\begin{align}
		\label{inequality:primal-solution-estimate}
		\left| f(x_{\rho,p_k}) + g(x_{\mu,p_k}) - v(D) \right| &\leq 2(1+2\sqrt{3})\epsilon, \\
		\label{inequality:primal-feasible-estimate}
		\left\| Ax_{\rho,p_k} - x_{\mu,p_k} \right\| &\leq \frac{2\epsilon}{R},
\end{align}
in the same number of iterations as needed in order to satisfy \eqref{opt-scheme:convergence-for-dual-problem}. Let $k:=k(\epsilon)$ be the smallest index with this property. By means of weak duality, i.\,e. $v(D) \leq v(P)$, \eqref{inequality:primal-solution-estimate} would imply that $f(x_{\rho,p_k}) + g(x_{\mu,p_k}) \leq v(P) + 2(1+2\sqrt{3})\epsilon$, which would further mean that $x_{\rho,p_k} \in \dom f$ and $x_{\mu,p_k} \in \dom g$ fulfilling \eqref{inequality:primal-solution-estimate} as well as \eqref{inequality:primal-feasible-estimate} can be seen as approximately optimal and feasible solutions to the primal optimization problem $(P)$ with an accuracy which is proportional to $\epsilon$.

Now let us prove the validity of the inequalities above. As $\nabla \theta_{\rho,\mu}(p_k) =Ax_{\rho,p_k} - x_{\mu,p_k}$, relation \eqref{inequality:primal-feasible-estimate} follows directly from \eqref{opt-scheme:convergence-for-dual-problem}. Thus, we have to prove only that \eqref{inequality:primal-solution-estimate} is true. To this aim, we notice first that, since $\theta_{\rho,\mu}(p_k) +v(D) \overset{\eqref{inequality:theta-rho-mu and theta relation}}{\leq} \theta(p_k) +v(D) \leq \epsilon$ and
\begin{align*}
		\theta_{\rho,\mu}(p_k) + v(D) &\overset{\eqref{inequality:theta-rho-mu and theta relation}}{\geq} \theta(p_k) - \rho D_f - \mu D_g + v(D)\\
		&\overset{\eqref{opt-problem-smoothing-parameters}}{=} \underbrace{\theta(p_k) + v(D)}_{\geq 0} -\frac{\epsilon}{2} \geq -\frac{\epsilon}{2},
\end{align*}
we have $\left| \theta_{\rho,\mu}(p_k) +v(D) \right| \leq \epsilon$. From \eqref{ineq:primal-dual} it follows
\begin{eqnarray*}
		\left| f(x_{\rho,p_k}) + g(x_{\mu,p_k}) -v(D) \right|
		&\leq& \left\| p_k \right\| \left\| \nabla \theta_{\rho,\mu}(p_k) \right\| + \epsilon + \rho D_f + \mu D_g  \\
		&\overset{\eqref{opt-problem-smoothing-parameters}}{\leq}& \left\| p_k \right\| \left\| \nabla \theta_{\rho,\mu}(p_k) \right\| + 2 \epsilon \\
		&\overset{\eqref{opt-scheme:convergence-for-dual-problem}}{\leq}& \frac{2\epsilon}{R} \left\| p_k \right\| +2\epsilon
\end{eqnarray*}
Further, in order to get an upper bound for $\left\| p_k \right\|$, we use that
\begin{align*}
	\left\| p_k \right\| = \left\| p_k + p_{DS}^* - p_{DS}^* \right\| \leq \left\| p_k -p_{DS}^* \right\| + \left\| p_{DS}^* \right\|
	\overset{\eqref{opt-scheme:norm-of-arguments}}{\leq} 2 \left\| p_{DS}^* \right\| \overset{\eqref{inequality:norm of optimal solution p_DS^*}}{\leq} 2\sqrt{3}R,
\end{align*}
and, finally, we obtain
$$\left| f(x_{\rho,p_k}) + g(x_{\mu,p_k}) -v(D) \right| \leq 4 \sqrt{3} \epsilon + 2 \epsilon = 2(2\sqrt{3}+1) \epsilon.$$

\subsection{Existence of an optimal solution}\label{subsectionExistence}

In this section we will study the convergence behavior of the primal sequences produced by the fast gradient method converge to an optimal solution of $(P)$ when $\epsilon \downarrow 0$. Let $(\epsilon_n)_{n\geq0} \subseteq \mathbb{R}_+$ be a decreasing sequence of positive scalars with $\lim_{n\rightarrow \infty}{\epsilon_n}=0$. For each $n\geq0$ we can make $k=k(\epsilon_n)$ iterations of the double smoothing algorithm  \eqref{doublesmooth:opt-sheme} with smoothing parameters $\rho_{\epsilon_n}$, $\mu_{\epsilon_n}$ and $\kappa_{\epsilon_n}$ given by \eqref{opt-problem-smoothing-parameters} in order to have \eqref{opt-scheme:convergence-for-dual-problem} satisfied. For $n \geq 0$ we denote
\begin{equation*}
	\bar{x}_n := x_{\rho_{\epsilon_n},p_{k(\epsilon_n)}}  \in \dom f \ \mbox{and} \ \bar{y}_n := x_{\mu_{\epsilon_n},p_{k(\epsilon_n)}}  \in \dom g.
\end{equation*}
Due to the boundedness of $\dom f$ and $\dom g$, there exist the subsequence of indices $(n_l)_{l\geq0} \subseteq (n)_{n\geq0}$, $\bar x \in \R^n$ and $\bar y \in \R^m$ such that
\begin{align*}
	\bar{x}_{n_l} \overset{l \rightarrow \infty}{\longrightarrow} \bar{x} \in \cl (\dom f) \ \mbox{and} \
	\bar{y}_{n_l} \overset{l \rightarrow \infty}{\longrightarrow} \bar{y} \in \cl (\dom g).
\end{align*}
In view of relation \eqref{inequality:primal-feasible-estimate} we obtain
\begin{align}
		\label{doublesmooth-feasible-estimate}
		0 \leq \left\| A\bar{x}_{n_l} - \bar{y}_{n_l} \right\| &\leq \frac{2\epsilon_{n_l}}{R},
\end{align}
for each $l\geq 0$. For $l \rightarrow +\infty$ in \eqref{doublesmooth-feasible-estimate} we get $A\bar{x}=\bar{y}$.
Furthermore, due to \eqref{inequality:primal-solution-estimate}, we have
\begin{align*}
	f(\bar{x}_{n_l}) + g(\bar{y}_{n_l}) \leq v(D) + 2(1+2\sqrt{3})\epsilon_{n_l} \quad \forall l \geq 0
\end{align*}
and, by using the lower semicontinuity of $f$ and $g$, we obtain
\begin{align*}
	f(\bar{x}) + g(A\bar{x}) \leq \liminf_{l \rightarrow \infty}{\left\{f(\bar{x}_{n_l}) + g(\bar{y}_{n_l})\right\}}
	\leq \lim_{l \rightarrow \infty}{\left\{v(D) + 2(1+2\sqrt{3})\epsilon_{n_l}\right\}} = v(D) \leq v(P).
\end{align*}
By taking into account that $v(P) < + \infty$, it follows that $\bar{x} \in \dom f$ and $A\bar{x} \in \dom g$, thus $\bar x$ is an optimal solution of the primal problem $(P)$.

\section{An example in image processing}\label{sectionExample}

In this section we are solving a linear inverse problem which arises in the field of signal and image processing by means of the double smoothing algorithm developed in the preceding sections. For a given matrix $A \in \mathbb{R}^{n \times n}$ describing a \textit{blur operator} and a given vector $b$ representing the \textit{blurred and noisy image} the task is to estimate the \textit{unknown original image} $x^*\in\mathbb{R}^n$ fulfilling
$$Ax=b.$$
To this end we solve the following nonsmooth $l_1$ regularized convex optimization problem
\begin{align*}
	\hspace{-1.8cm}(P) \quad \quad \inf_{x \in S}{\left\{ \left\| Ax-b \right\|_1 + \lambda \left\| x \right\|_1\right\}},
\end{align*}
where $S\subseteq \mathbb{R}^n$ is an $n$-dimensional cube representing the range of the pixels and $\lambda > 0$ is the regularization parameter. The problem to be solved can be equivalently written as
\begin{align*}
	\hspace{-1.8cm}(P) \quad \quad \inf_{x \in \mathbb{R}^n}{\left\{ f(x) + g(Ax)\right\}},
\end{align*}
for $f:\mathbb{R}^n \rightarrow \overline{\mathbb{R}}$, $f(x)=\lambda\left\| x \right\|_1 + \delta_{S}(x)$ and $g:\mathbb{R}^n \rightarrow \overline{\mathbb{R}}$, $g(y)=\left\|y-b \right\|_1+\delta_S(y)$ (one has that $A(S) \subseteq S$, since for $x \in S$ the pixels of the blurred picture $Ax$ have naturally the same range). Thus both functions $f$ and $g$ are proper, convex and lower semicontinuous and have bounded effective domains.

Since each pixel furnishes a greyscale value which is between $0$ and $255$, a natural approach for the convex set $S$ would be the $n$-dimensional cube $\left[0,255\right]^n \subseteq \mathbb{R}^n$. In order to reduce the Lipschitz constants  which appear in the developed approach, we scale all the pictures used within this section so that each of their  pixels ranges in the intervall $\left[0,\frac{1}{10}\right]$.

In this section we concretely look at the $256 \times 256$ \textit{cameraman test image}, which is part of the image processing toolbox in Matlab. The dimension of the vectorized and scaled cameraman test image is $n=256^2=65536$. By making use of the Matlab functions {\ttfamily imfilter} and {\ttfamily fspecial}, this image is blurred as follows:
\begin{lstlisting}[numbers=left,numberstyle=\tiny,frame=tlrb,showstringspaces=false]
H=fspecial('gaussian',9,4);   % gaussian blur of size 9 times 9
                              % and standard deviation 4
B=imfilter(X,H,'conv','symmetric');  % B=observed blurred image
                                     % X=original image		
\end{lstlisting}
In row $1$ the function {\ttfamily fspecial} returns a rotationally symmetric Gaussian lowpass filter of size $9 \times 9$ with standard deviation $4$. The entries of $H$ are nonnegative and their sum adds up to $1$. In row $3$ the function {\ttfamily imfilter} convolves the filter $H$  with the image $X\in \mathbb{R}^{256 \times 256}$ and outputs the blurred image $B\in \mathbb{R}^{256 \times 256}$ . The boundary option "symmetric" avoids dark edges for the blurred picture $B$ which normally appears after a convolution (provided that $X$ and $B$ have same dimensions).

Thanks to the rotationally symmetric filter $H$, the linear operator $A\in\mathbb{R}^{n \times n}$ given by the Matlab function {\ttfamily imfilter} is symmetric, too. Since each entry in $B$ can be seen as a convex combination of elements in $X$ with coefficients in $H$, we have $A(S)\subseteq S$.
The norm $\left\| A \right\|^2$ is not explicitly given and is estimated by $1$. After adding a zero-mean white Gaussian noise with standard deviation $10^{-4}$, we obtain the blurred and noisy image $b \in \mathbb{R}^n$ which is shown in Figure \ref{fig:cameraman}.
\begin{figure}[ht]
	\centering
	\includegraphics*[page=1, viewport= 58 318 560 553, width=0.8\textwidth]{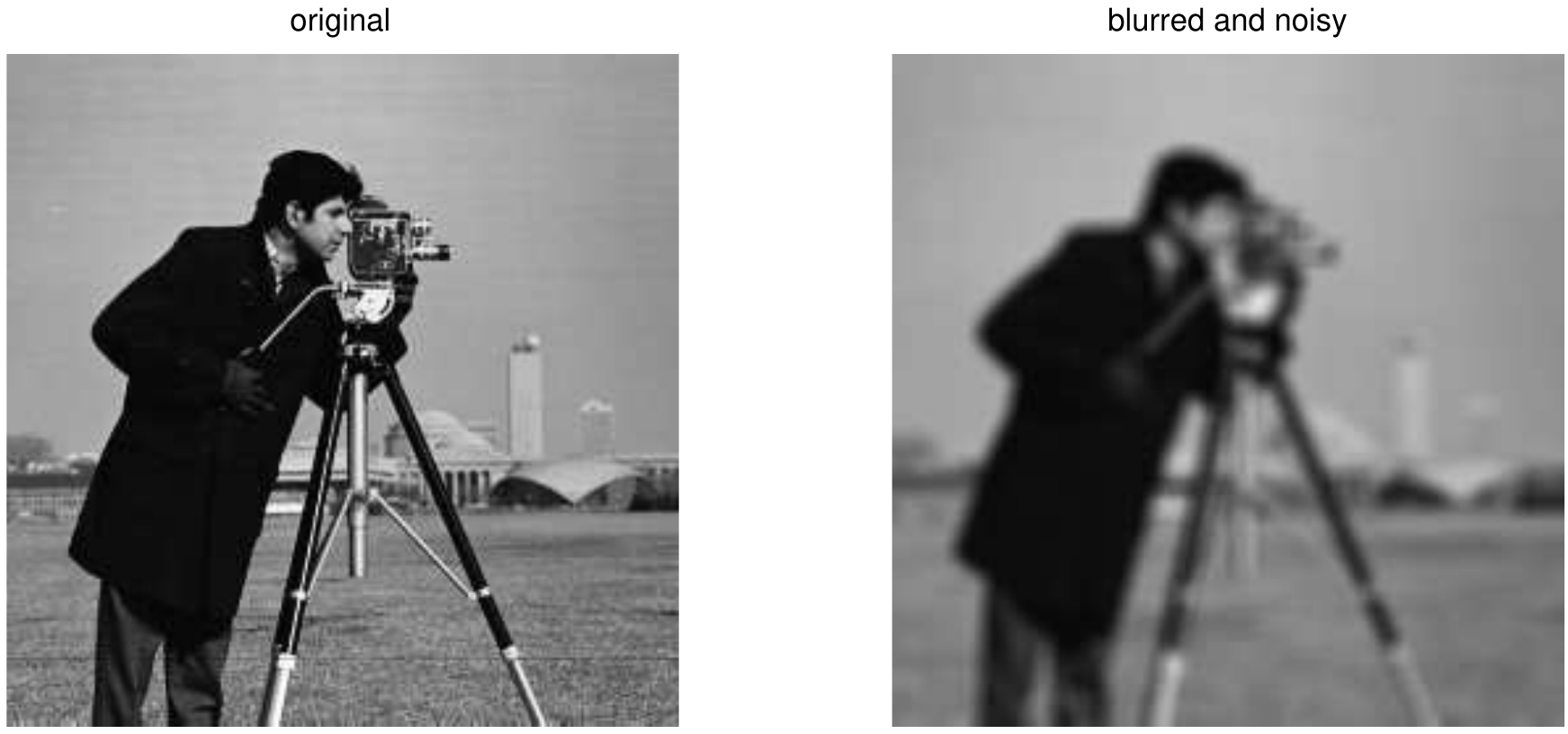}
	\caption{The $256 \times 256$ cameraman test image}
	\label{fig:cameraman}
\end{figure}

One should also notice that, as both functions occurring in the formulation of $(P)$ are nondifferentiable, the classical iterative shrinkage thresholding algorithm and its variants (see \cite{BeckTeboulle09, BeckTeboulle10, Daubechiesetal04}) cannot be taken into account for solving this optimization problem. Indeed, in this situation the double smoothing technique is our first choice for solving $(P)$ with an optimal first-order method.

The dual optimization problem in minimization form is
\begin{align*}
	\hspace{-1.8cm}(D) \quad \quad -\inf_{p \in \mathbb{R}^n}{\left\{f^*(A^*p)+g^*(-p)\right\}}
\end{align*}
and, due to the fact that $x':=\frac{1}{20}\mathbbm{1}^n \in \ri(S) \cap A(\ri(S))$, it has an optimal solution (see, for instance, \cite{Bot10,BotGradWanka09}). By taking into consideration \eqref{opt-problem-smoothing-parameters}, the smoothing parameters are taken
\begin{align}
	\label{example-smoothing-parameters}
	\rho = \frac{\epsilon}{4D_f}, \quad \mu = \frac{\epsilon}{4D_g}, \quad \kappa = \frac{\epsilon}{2R^2},
\end{align}
for $D_f = D_g = \sup{\left\{ \frac{\left\|x\right\|^2}{2} : x \in \left[0,\frac{1}{10}\right]^n \right\}} = 327.68$ and $R=0.05$, while the accuracy is chosen to be $\epsilon=0.01$.

In the following we show that the proximal points can be exactly calculated in each iteration of the algorithm, due to the fact that they occur as optimal solutions of some separable convex optimization problems. Indeed, since for $k \geq 0$
$${}^{\frac{1}{\rho}}f\left(\frac{A^*w_k}{\rho}\right) = \inf_{x\in \mathbb{R}^n}{\left\{ f(x) + \frac{\rho}{2}\left\|\frac{A^*w_k}{\rho} - x\right\|^2 \right\}} = \inf_{x\in \left[0,\frac{1}{10}\right]^n}{\left\{ \lambda \left\|x\right\|_1 + \frac{\rho}{2}\left\|\frac{A^*w_k}{\rho} - x\right\|^2 \right\}},$$
the proximal point of $f$ of parameter $\tfrac{1}{\rho}$ at  $\frac{A^*w_k}{\rho}$ fulfills
\begin{align*}
	 x_{\rho,w_k} = \argmin_{x\in \left[0,\frac{1}{10}\right]^n}{\left\{ \sum_{i=1}^{n}{\left[ \lambda \left|x_i\right| + \frac{\rho}{2} \left(\frac{(A^*w_k)_i}{\rho} - x_i \right)^2\right]}  \right\}}
\end{align*}
and its calculation requires the solving of the following one-dimensional convex optimization problem for $i=1,\ldots,n$:
\begin{align*}
	 \inf_{x_i\in \left[0,\frac{1}{10}\right]}{\left\{ \lambda x_i + \frac{\rho}{2} \left(\frac{(A^*w_k)_i}{\rho} - x_i \right)^2 \right\}},
\end{align*}
which has as unique optimal solution $\mathcal{P}_{\left[0,\frac{1}{10}\right]}\left( \frac{1}{\rho} \left( (A^*w_k)_i - \lambda \right) \right)$.
Thus,
$$x_{\rho,w_k} = \mathcal{P}_{\left[0,\frac{1}{10}\right]^n}\left( \frac{1}{\rho} \left( A^*w_k - \lambda \mathbbm{1}^n \right) \right).$$

On the other hand, since for $k \geq 0$
\begin{align*}
{}^{\frac{1}{\mu}}g\left(-\frac{w_k}{\mu}\right) &= \inf_{x\in \mathbb{R}^n}{\left\{ g(x) + \frac{\mu}{2}\left\|-\frac{w_k}{\mu} - x\right\|^2 \right\}} = \inf_{x\in \left[0,\frac{1}{10}\right]^n}{\left\{ \left\|x-b\right\|_1 + \frac{\mu}{2}\left\|-\frac{w_k}{\mu} - x\right\|^2 \right\}} \\
	&= \inf_{x\in \left[0,\frac{1}{10}\right]^n}{\left\{ \sum_{i=1}^n{\left[ \left| x_i - b_i \right| + \frac{\mu}{2}\left( -\frac{(w_k)_i}{\mu} - x_i\right)^2 \right]} \right\}},
\end{align*}
the calculation of the proximal point of $g$ of parameter $\tfrac{1}{\mu}$ at  $\frac{-w_k}{\mu}$  requires the solving of the following one-dimensional convex optimization problem for $i=1,\ldots,n$:
\begin{align*}
	 \inf_{x_i\in \left[0,\frac{1}{10}\right]}{\left\{ \left| x_i - b_i \right| + \frac{\mu}{2}\left( -\frac{(w_k)_i}{\mu} - x_i\right)^2 \right\}}.
\end{align*}
\begin{figure}[ht]	
	\centering
	\includegraphics*[page=1, viewport= 104 469 506 662, width=0.8\textwidth]{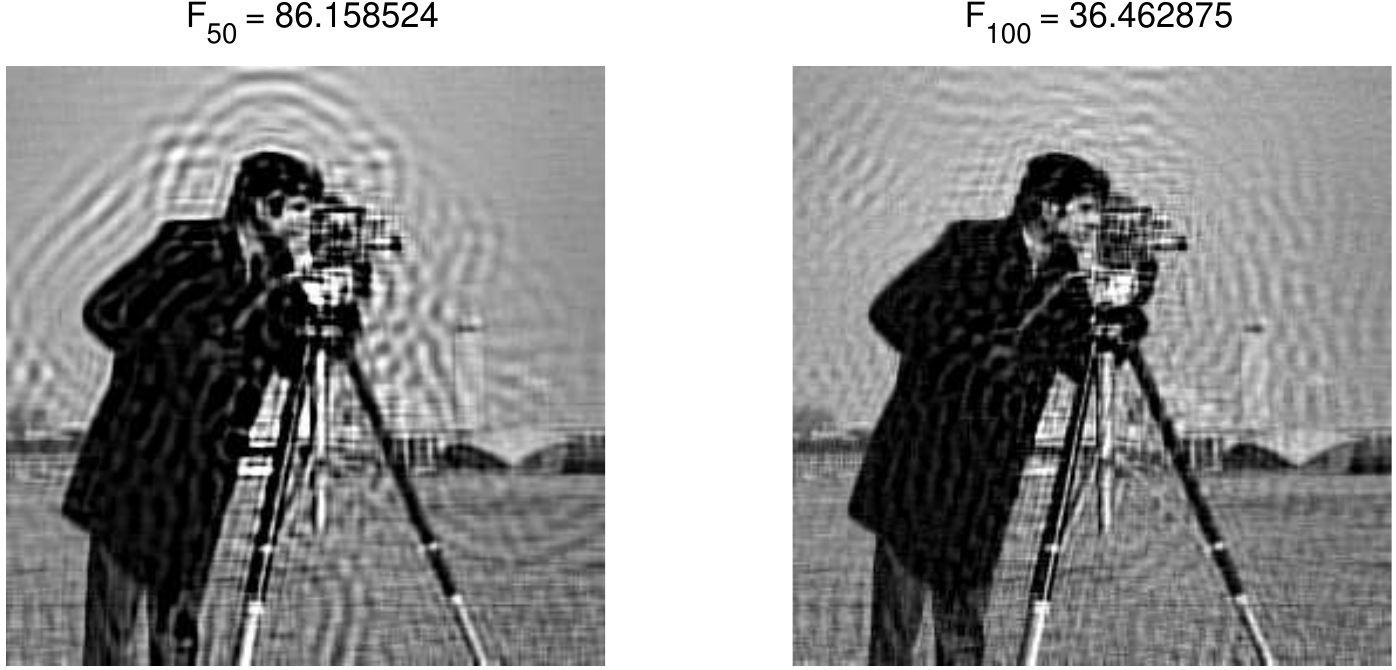}
	\includegraphics*[page=1, viewport= 104 216 506 418, width=0.8\textwidth]{fig_4cameramen.pdf}
	\caption{Iterations of the double smoothing algorithm}
	\label{fig:4cameramen}	
\end{figure}

 For a fixed $k \geq 0$ we consider for $i =1,...,n$ the function $h_i:\mathbb{R} \rightarrow \mathbb{R}$, $h_i(z)=\left| z - b_i \right| + \frac{\mu}{2}\left( -\frac{(w_k)_i}{\mu} - z\right)^2$. For for $i=1,...,n$ the optimal solution of the above problem is  the projection of the unique global minimum (cf. \cite[Proposition A.8 and Proposition B.10]{Bertsekas99}) $z_i$ of $h_i$ on $\left[0,\frac{1}{10}\right]$. For $i=1,...,n$ we have
\begin{align*}
	 0 \in \partial h_i(z_i) = \partial \left( \left| \cdot - b_i \right| + \frac{\mu}{2}\left( -\frac{(w_k)_i}{\mu} - \cdot \right)^2 \right)(z_i) = \partial \left( \left| \cdot - b_i \right| \right)(z_i) - \mu \left( -\frac{(w_k)_i}{\mu} - z_i \right),
\end{align*}
which is equivalent to
\begin{align*}
	-(w_k)_i \in \partial \left( \left| \cdot - b_i \right| \right)(z_i) + \mu z_i
	= \left\{ \begin{aligned} 1 + \mu z_i &: z_i > b_i \\ \left[ -1+\mu b_i,1+\mu b_i \right] &: z_i = b_i \\ -1 +\mu z_i &: z_i < b_i \end{aligned} \right..
\end{align*}
Hence, the unique global minimum $z_i$ can be calculated as follows
\begin{align*}
	z_i
	= \left\{ \begin{aligned} -\frac{(w_k)_i+1}{\mu} &: (w_k)_i < -1 -\mu b_i \\ b_i \hspace{0.7cm} &: -1-\mu b_i \leq (w_k)_i \leq 1-\mu b_i \\ \frac{1-(w_k)_i}{\mu} &: (w_k)_i > 1-\mu b_i \end{aligned} \right. .
\end{align*}
All in all, the proximal point of $g$ of parameter $\tfrac{1}{\mu}$ at  $\frac{-w_k}{\mu}$ is for $z=(z_1,...,z_n)^T$ given by
\begin{align*}
x_{\mu,w_k} = \mathcal{P}_{\left[0,\frac{1}{10}\right]^n}\left( z \right).
\end{align*}

\begin{figure}[ht]	
	\centering
	\includegraphics*[page=1, viewport= 26 249 583 596, width=1\textwidth]{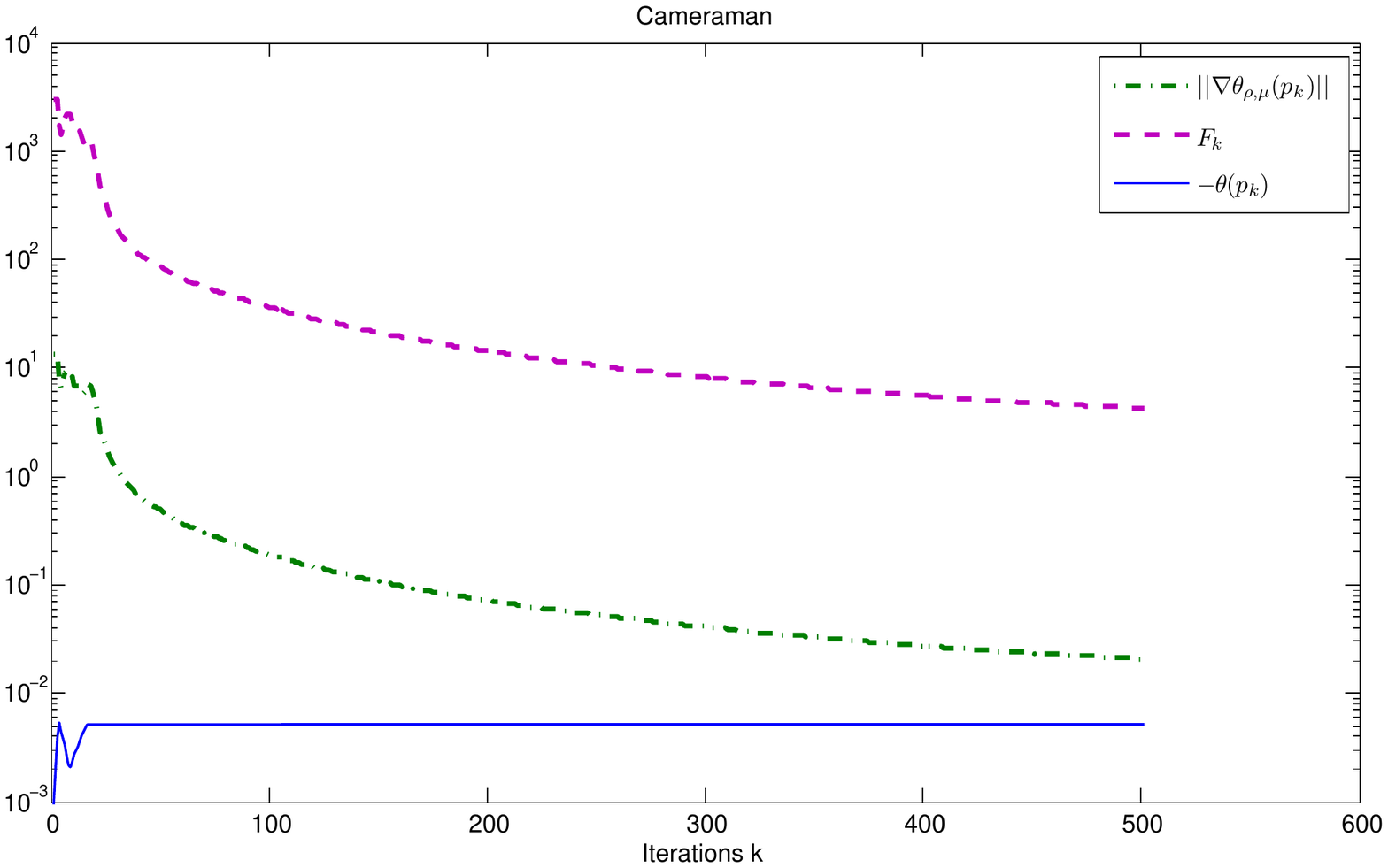}
	\caption{Convergence to an approximately optimal and feasible primal solution}
	\label{fig:l1_convergence}
\end{figure}

The iterations 50, 100, 200 and 500 of the double smoothing iterative scheme are shown in Figure \ref{fig:4cameramen} for $\lambda=2$e-$6$ and $F_k := f(x_{\rho,p_k}) + g(Ax_{\rho,p_k})$. The decrease of $F_k$ and $\left\|Ax_{\rho,p_k} - x_{\mu,p_k}\right\|$ can be seen in Figure \ref{fig:l1_convergence}. The function values of $-\theta(p_k)$ are shown in the latter as well.

\section{Conclusions}

The subject of this paper can be summarized as a development of a first-order method for solving unconstrained nondifferentiable convex optimization problems in finite dimensional spaces having as objective the sum of a convex function with the composition of another convex function with a linear operator.
The provided method assumes the minimization of the doubly regularized Fenchel dual objective and allows to reconstruct an approximately optimal primal solution in $O\left(\frac{1}{\epsilon} \ln\left( \frac{1}{\epsilon}\right)\right)$ iterations which outperforms the classical subgradient approach.

\end{document}